\documentclass[12pt]{amsart}
\usepackage{amsmath,amssymb,mathrsfs,color}
\usepackage[title]{appendix}
\usepackage{amssymb}
\usepackage{enumitem}

     \usepackage{hyperref}

\allowdisplaybreaks

\def\T{{\mathbb T}}
\def\Z{{\mathbb Z}}
\def\R{{\mathbb R}}
\def\P{{\mathbb P}}
\def\E{{\mathbb E}}

\def\eps{\epsilon}
\newtheorem{thm}{Theorem}[section]

\newtheorem{lem}[thm]{Lemma}
\newtheorem{prop}[thm]{Proposition}

\theoremstyle{definition}

\theoremstyle{remark}
\newtheorem{rem}[thm]{Remark}
\newtheorem{exam}[thm]{Example}
\numberwithin{equation}{section}

\newcommand{\Lip}{\operatorname{Lip}}
\newcommand{\rmd}{{\rm d}}

\newcommand{\cB}{{\mathcal B}}
\newcommand{\cV}{{\mathcal V}}
\newcommand{\cG}{{\mathcal G}}
\newcommand{\cX}{{\mathcal X}}
\newcommand{\cF}{{\mathcal F}}
\newcommand{\cQ}{{\mathcal Q}}
\newcommand{\cW}{{\mathcal W}}
\newcommand{\tF}{{\widetilde F}}
\newcommand{\tW}{{\widetilde W}}
\newcommand{\tY}{{\widetilde Y}}
\newcommand{\bmu}{{\bar \mu}}
\newcommand{\bpi}{{\bar \pi}}
\newcommand{\bF}{{\bar F}}
\newcommand{\bY}{{\overline Y}}
\newcommand{\bW}{{\overline W}}
\newcommand{\bv}{{\bar v}}
\newcommand{\hv}{{\hat v}}

\newcommand{\infYj}{{\textstyle{\inf_{Y_j}}}}
\newcommand{\infbYj}{{\textstyle{\inf_{\bY\!_j}}}}

\begin{document}

\title[Wasserstein convergence rates in the invariance principle]{Wasserstein convergence rates in the invariance principle for nonuniformly hyperbolic flows}

\author{Ian Melbourne}
\address{I. Melbourne: Mathematics Institute, University of Warwick, Coventry, CV4 7AL, UK}
\email{i.melbourne@warwick.ac.uk}

\author{Zhe Wang}
\address{Z. Wang: School of Mathematical Sciences, Dalian University of Technology, Dalian
116024, P. R. China}
 \email{zwangmath@hotmail.com}

\date{25 September 2025; Updated 6 November 2025}

\subjclass[2010]{37C10, 37A50, 60F17, 37D20}

\keywords{Nonuniformly hyperbolic flow, suspension flow, invariance principle, rate of convergence, Wasserstein distance}

\begin{abstract}
We obtain $q$-Wasserstein convergence rates in the invariance principle for nonuniformly hyperbolic flows,
where $q\ge1$ depends on the degree of nonuniformity. Utilizing a martingale-coboundary decomposition for nonuniformly expanding semiflows, we extend techniques from the discrete-time setting to the continuous-time case. Our results apply to uniformly hyperbolic (Axiom A) flows, nonuniformly hyperbolic flows that can be modelled by suspensions over Young towers with exponential tails (such as dispersing billiard flows and the classical Lorenz attractor), and intermittent solenoidal flows.
\end{abstract}

\maketitle

\section{Introduction}
\label{sec:intro}

\setcounter{equation}{0}

The statistical properties of uniformly expanding/hyperbolic maps and flows are by now well understood, including the central limit theorem (CLT) for H\"older observables \cite{B75,R73,S60} and the almost sure invariance principle \cite{DP84}.
The latter result implies both the CLT and its functional form, commonly referred to as the weak invariance principle (WIP).
Subsequently, there has been great interest in such statistical properties for nonuniformly expanding/hyperbolic maps and flows (see for example~\cite{GM16,G04,HK82,MN05,MT04,MV16,Y98,Y99}). In addition, sharp Berry-Esseen estimates (convergence rates in the CLT) for nonuniformly expanding maps were obtained in~\cite{G05}.

In recent years, interest has turned to quantifying convergence rates in the WIP for dynamical systems, using metrics such as the L\'evy-Prokhorov and Wasserstein distances.
Antoniou and Melbourne~\cite{AM19} established convergence rates in the L\'evy-Prokhorov distance for nonuniformly hyperbolic maps.
More recently, Liu and Wang obtained $q$-Wasserstein convergence rates for deterministic hyperbolic maps~\cite{LW24a} and for sequential dynamical systems~\cite{LW24b}. Dedecker, Merlev\`ede and Rio \cite{DMR24} provided an improved rate in $2$-Wasserstein distance.
Regarding the multidimensional WIP (i.e.\ $\R^d$-valued observables with $d>1$), Paviato~\cite{P24} derived convergence rates in the L\'evy-Prokhorov distance and the $1$-Wasserstein distance
for nonuniformly expanding/hyperbolic maps and flows.
Subsequently, Fleming-V\'azquez~\cite{F25} obtained improved convergence rates in the $q$-Wasserstein distance for nonuniformly hyperbolic maps.

However, in the continuous-time literature, the study of convergence rates in the CLT and the WIP remains limited.
To the best of our knowledge, only three works have addressed this issue. P\`ene \cite{P02,P07} was the first to provide
convergence rates in the CLT for dispersive billiards with finite horizon. More recently, the aforementioned results of Paviato \cite{P24} on convergence rates in the WIP cover both discrete and continuous time.

A standard approach for deriving the CLT or WIP for flows from the corresponding result for maps is via inducing (see for example~\cite{GM16,MT04,R73}).
Indeed, the error rates in the CLT for flows in~\cite{P02,P07} are obtained in this way, though the rate for flows is weaker than the rate for maps. In this paper we pursue a different method, following Paviato~\cite{P24}, working with the flow directly and obtaining exactly the same error rate as for maps. The key technique in~\cite{P24} is a secondary martingale-coboundary decomposition for nonuniformly expanding semiflows
extending the work of \cite{KKM18} for nonuniformly expanding maps.

For scalar observables, Paviato~\cite{P24} obtains the Prokhorov convergence rate $n^{-\frac{1}{4}}(\log n)^{\frac{3}{4}}$ in the WIP for uniformly hyperbolic maps and flow, and the rate $n^{-\frac14+\delta}$ for nonuniformly hyperbolic flows where $\delta>0$ depends on the degree of nonuniformity $p\in(2,\infty)$ and $\delta\to0$ as $p\to\infty$.
The rates for nonuniformly hyperbolic flows in~\cite{P24} are the same as those in~\cite{AM19} for maps. The paper~\cite{P24} also considers the $1$-Wasserstein distance for vector-valued observables of nonuniformly hyperbolic maps and flows, obtaining a rate $n^{-\frac16}(\log n)^{\frac13}$ as $p\to 3^-$
(independent of the dimension of the observable). However, the obtained rate does not improve for $p\ge3$.

In this paper, as in~\cite{LW24a}, we restrict to scalar observables but consider convergence rates for the $q$-Wasserstein distance $1\le q\le \frac{p}{2}$, $p>2$. We establish the convergence rate
$n^{-\frac14}(\log n)^{\frac12}$ for all nonuniformly hyperbolic flows with $p\ge4$.
(In the process, we improve the rates obtained in~\cite{LW24a} for maps.)
Our results are new for $q\ge2$ and also significantly improve the results in~\cite{P24} for $q=1$ in the scalar observable case.

Our results are applicable to uniformly hyperbolic (Axiom A) flows  and a class of nonuniformly hyperbolic flows modelled by suspensions
over Young towers with exponential tails, such as  planar periodic Lorentz gases with finite horizon, Lorentz gases with cusps and the classical Lorenz attractor. In such examples, we obtain the convergence rate $n^{-\frac14}(\log n)^{\frac12}$ in the $q$-Wasserstein distance for all $q\ge1$, see Section~\ref{sec:mainflows}.
Our results also cover
intermittent semiflows (Example~\ref{ex:lsv}) and intermittent solenoidal flows (Example~\ref{ex:solenoid}).

The remainder of this paper is organized as follows. In Section~\ref{sec:main}, we introduce nonuniformly expanding semiflows and state the main results.
Section~\ref{sec:MCD} summarizes some recent results of \cite{P24} on martingale approximation.
The proof of the main result is given in Section~\ref{sec:exp2}. In Section~\ref{sec:NUHflow}, we extend the convergence rate to nonuniformly hyperbolic flows.

\subsubsection*{Notation:}
Throughout the paper, we use $1_A$ to denote the indicator function of measurable set $A$
 and $\|\cdot\|_{L^p}$ means the $L^p$-norm.
As usual, $a_n=O(b_n)$ means that there exists a constant $C>0$ such that $|a_n|\le C |b_n|$ for all $n\ge 1$.
For simplicity we write $C$ to denote constants independent of $n$, and $C$ may change from line to line.
We use $\to_{w}$ to denote weak convergence in the sense of probability measures and $\to_{d}$ means convergence in distribution.
We use $\P_X$ to denote the law/distribution of random variable $X$ and use $X=_d Y$ to mean $X, Y$ sharing the same distribution.

We denote by $C[0,1]$ the space of all continuous functions on $[0,1]$ equipped with the supremum distance $d_C$, that is
\[
d_C(x,y):=\sup_{t\in [0,1]}|x(t)-y(t)|, \quad x,y\in C[0,1].
\]

Let $(M,d)$ be a bounded metric space and let $\eta\in(0,1]$.
We denote by $C^\eta(M)$ the space of $C^\eta$ functions $v:M\to\R$ with
finite H\"older norm $\|v\|_{C^\eta}=|v|_\infty+|v|_{C^\eta}$, where
\[
|v|_\infty=\sup_{x\in M}|v(x)|, \qquad |v|_{C^\eta}=\sup_{x,y\in M\atop x\neq y}\frac{|v(x)-v(y)|}{d(x,y)^\eta}.
\]

\section{Statement of the main results}
\label{sec:main}

\setcounter{equation}{0}

\subsection{Nonuniformly expanding maps}
\label{sec:NUE}

Let $(X,d)$ be a bounded metric space with Borel probability measure $\rho$ and let $T:X\to X$ be a nonsingular, ergodic transformation. Suppose that $Y$ is a subset of $X$ with $\rho(Y)>0$, and $\{Y_j\}$ is an at most countable measurable partition of $Y$ with $\rho(Y_j)>0$.
Let $r:Y\to \Z^+$ be constant on each $Y_j$ and such that $T^{r(y)}y\in Y$ for all $y\in Y$. We call $r$ the {\em return time} and $F=T^r:Y\to Y$ the {\em induced map}.

We suppose that there are constants $\lambda>1$, $C\ge 1$, $\eta\in (0,1]$ such that for each $j\ge 1$,
\begin{enumerate}
  \item $F|_{Y_j}:Y_j\to Y$ is a (measure-theoretic) bijection,
  \item $d(Fx,Fy)\geq\lambda d(x,y)$ for all $x,y\in Y_j$,
  \item $d(T^{\ell}x,T^{\ell}y)\leq C d(Fx,Fy)$ for all $x,y\in Y_j$, $0\leq \ell < r(j)$,
  \item $g_j=\frac{\mathrm{d}\rho|_{Y_j}}{\mathrm{d}\rho|_{Y_j}\circ F}$ satisfies $|\log g_j(x)-\log g_j(y)|\leq Cd(Fx,Fy)^{\eta}$ for all $x,y\in Y_j$.
\end{enumerate}
Assume that $r\in L^p(Y)$ for some $p\ge1$; then we call $T:X\to X$ a \emph{nonuniformly expanding map of order $p$}.

As a consequence of conditions (1), (2) and (4), the map $F$ is a \emph{(full-branch) Gibbs-Markov map}~\cite{AD01}.
It is standard that there exists a unique absolutely continuous $F$-invariant probability measure $\mu_Y$ on $Y$.

\subsection{Nonuniformly expanding semiflows}
\label{sec:NUEsemi}

Let $\Psi_t:M\to M$ be a semiflow on a bounded metric space $(M,d)$, satisfying $\Psi_0=\text{Id}$ and $\Psi_{t+s}=\Psi_t\circ\Psi_s$ for $s,t\ge 0$. We assume that there exists $C>0$ such that
\begin{equation}\label{eq:con}
d(\Psi_tx, \Psi_ty)\le Cd(x,y) \quad\text{~for~all~} t\in [0,1],\ x,y\in M,
\end{equation}
and
\begin{equation}\label{eq:tim}
d(\Psi_tx, \Psi_sx)\le C|t-s|  \quad\text{~for~ all~} s,t\ge 0, \ x\in M.
\end{equation}

Let $X\subset M$  be a Borel set and define the first return time $h:X\to \R^{+}$, $h(x)=\inf\{t>0:\Psi_tx\in X\}$.
We assume that $h\in C^\eta(X)$ for some fixed $\eta\in(0,1]$, and that $\inf h\ge 1$. Such a function $h$ is often called a {\em roof function}.
The map $T=\Psi_h:X\to X$ is called a {\em Poincar\'e map}. If $T$ is a nonuniformly expanding map of order $p\ge1$ as described in  Subsection~\ref{sec:NUE},
then we call the semiflow $\Psi_t:M\to M$ a {\em nonuniformly expanding semiflow of order $p$}.

Define the induced roof function
\[
\varphi:Y\to [1,\infty), \qquad \varphi(y)=\sum_{i=0}^{r(y)-1}h(T^iy).
\]
Since $r\in L^p(Y)$ and $\varphi\le |h|_\infty r$,
we have $\varphi\in L^p(Y)$.
We define the suspension $Y^{\varphi}=\{(y,u)\in Y\times \R: 0\le u\le \varphi(y)\}/\sim$,
where $(y,\varphi(y))\sim (Fy,0)$. The suspension semiflow $F_t:Y^{\varphi}\to Y^{\varphi}$ is given by $F_t(y,u)=(y,u+t)$ computed modulo identifications.
Also, define the ergodic $F_t$-invariant probability measure $\mu^\varphi=(\mu_Y\times \text{Lebesgue})/\bar \varphi$,
where $\bar\varphi=\int_Y\varphi\,\rmd \mu_Y$.

We have a projection $\pi_M:Y^{\varphi}\to M$ given by $\pi_M(y,u)=\Psi_u y$ and it is a semiconjugacy between $F_t$ and $\Psi_t$,
satisfying $\Psi_t\circ \pi_M=\pi_M\circ F_t$. Hence $\mu_M=(\pi_M)_{*}\mu^{\varphi}$ is an ergodic $\Psi_t$-invariant probability measure on $M$.

\subsection{Statement of the main results}
\label{sec:statemain}

Let $C_0^{\eta}(M)=\{v\in C^{\eta}(M): \int_{M} v \,\rmd \mu_M=0\}$.
Given $v\in C_0^{\eta}(M)$, we denote $v_t=\int_{0}^{t}v\circ \Psi_s\,\rmd s$ and define the continuous processes $W_n\in C[0,1]$, $n\ge1$, as
\[
W_n(t):=\frac{1}{\sqrt{n}}\int_0^{nt}v\circ\Psi_s\,\rmd s, \quad t\in [0,1].
\]

The following result is standard; see \cite{KM16,KKM22,MN05,MT04,MT12} for details.
\begin{prop}\label{prop:summ}
Let $\Psi_t:M\to M$ be a nonuniformly expanding semiflow of order $p>2$ and $v\in C_0^{\eta}(M)$. Then

$(a)$ {\em(CLT)} The limit $\sigma^2=\lim_{t\to\infty}t^{-1}\int_{M}v_t^2\,\rmd \mu_M$ exists, and $t^{-1/2}v_t\to_d N(0,\sigma^2)$ as $t\to\infty$.

$(b)$ {\em(WIP)} $W_n\to_w W$ in $C[0,1]$ as $n\to\infty$, where $W$ is a Brownian motion with mean zero and variance $\sigma^2>0$.

$(c)$ {\em(Moment bounds)} There exists $C>0$ such that $\big\|\sup_{t\in [0,K]}|v_t|\,\big\|_{L^{2(p-1)}(M)}\le CK^{1/2}\|v\|_{C^\eta}$ for all $K>0$.
\end{prop}

\begin{proof}
Define the induced observable $\tilde v: X\to \R$,
$\tilde v= \int_0^{h} v\circ \Psi_u\,\rmd u$.
Since $v\in\ C_0^{\eta}(M)$ and $h\in C^\eta(X)$, it follows easily from~\eqref{eq:con} that $\tilde v\in C^{\eta}(X)$ with $\int_X\tilde v\,d\mu=0$.
It is well known, see e.g.\ \cite[Corollary~2.13]{KKM18}, that the WIP holds for $\tilde v:X\to \R$.
By a standard inducing argument (see for example~\cite[Proposition~5.7 and Theorem~5.8]{KKM22}), the WIP holds for $v:M\to\R$.

Similarly, by e.g.\ \cite[Corollary~2.10]{KKM18},
\(
\big\| \max_{k\le n} | \sum_{j=0}^{k-1}\tilde v\circ T^j | \big\|_{L^{2(p-1)}(X)} \le C\|v\|_{C^\eta} n^{1/2}.
\)
Since the roof function $h$ is bounded below, the moment estimate for the semiflow follows easily by a standard argument
(see for example,~\cite[proof of Lemma~4.1]{MT12} or~\cite[Section~7.2]{KM16}).
\end{proof}

In the present paper, we consider the Wasserstein distance to metrize weak convergence. For $q\ge 1$, we denote by $\cW_q(\mu,\nu)$ the Wasserstein distance between the distributions $\mu$ and $\nu$ on a Polish space $(\cX,d)$ (see \cite[Definition~6.1]{V09}):
\[
\cW_q(\mu,\nu)= \inf \big\{ [\E\, {d(X,Y)}^q]^{1/q};\, \hbox{law} (X)=\mu,\, \hbox{law} (Y)=\nu \big\}.
\]
\begin{prop}[{{\cite[Definition 6.8]{V09}}}]\label{prop:lid}
We have that $\lim_{n\to\infty}\cW_q(\mu_n, \mu)=0$ if and only if the following two conditions hold:
\begin{enumerate}
  \item $\mu_n\to_{w}\mu$ as $n\to\infty$;
  \item $\lim_{n\to\infty}\int_{\cX} d(x,x_{0})^q\,\rmd\mu_n(x)= \int_{\cX} d(x,x_{0})^q\,\rmd\mu(x)$ for some $($thus any$)$ $x_{0}\in\cX$.
\end{enumerate}
 In particular, if the metric $d$ is bounded, then the convergence with respect to $\cW_q$ is equivalent to the weak convergence in condition (1).
\end{prop}
In the following, we use the abbreviation $\cW_q(X,Y)$ to mean $\cW_q(\P_X, \P_Y)$.
\begin{thm}\label{thm:exp1}
Let $\Psi_t:M\to M$ be a nonuniformly expanding semiflow of order $p>2$ and $v\in C_0^{\eta}(M)$. Then $\lim_{n\to\infty}\cW_q(W_n,W)=0$ for all $1\le q< 2(p-1)$.
\end{thm}

\begin{proof}
The proof is essentially the same as that of \cite[Theorem~3.3]{LW24a}, and is included here for completeness. It follows from Proposition~\ref{prop:summ}$(c)$ that $W_n$ has a finite moment of order $2(p-1)$. This together with the fact that $W_n\to_{w} W$ in Proposition~\ref{prop:summ}$(b)$ implies that for each $q<2(p-1)$,
\[
\lim_{n \to \infty}\E\sup_{t\in [0,1]}|W_n(t)|^q=\E\sup_{t\in [0,1]}|W(t)|^q
\]
by \cite[Theorem 4.5.2]{Chung}. On the other hand, using the fact that $W_n: M\to C[0,1]$ and the definition of pushforward measures, we have
\[
\int_{C[0,1]} d_C(x,0)^q \, \rmd \mu_M\circ W_n^{-1}(x)=\int_{M}  \sup_{t\in [0,1]}|W_n(t,\omega)|^q \,\rmd \mu_M(\omega)=\E\sup_{t\in [0,1]}|W_n(t)|^q;
\]
similarly,
\[
\int_{C[0,1]} d_C(x,0)^q \, \rmd \mu_M\circ W^{-1}(x)=\E\sup_{t\in [0,1]}|W(t)|^q.
\]
Hence
\[
\lim_{n \to \infty}\int_{C[0,1]}d_C(x,0)^q \, \rmd \mu_M\circ W_n^{-1}(x)=\int_{C[0,1]}d_C(x,0)^q \, \rmd \mu_M\circ W^{-1}(x).
\]
By taking $\mu_n=\mu_M\circ W_n^{-1}, \mu=\mu_M\circ W^{-1}$ and $x_0=0$ in Proposition~\ref{prop:lid} and using the fact that $W_n\to_{w} W$ in $C[0,1]$, the result follows.
\end{proof}

Our main result for semiflows is the following:
\begin{thm}\label{thm:exp2}
Let $\Psi_t:M\to M$ be a nonuniformly expanding semiflow of order $p>2$ and $v\in C_0^{\eta}(M)$.
Then there is a constant $C>0$ such that
\begin{align*}
  \cW_{\frac{p}{2}}(W_n,W)\le
  \begin{cases}
Cn^{-\frac12+\frac{1}{p}}(\log n)^\frac12, & 2<p<4,\\
  Cn^{-\frac14}(\log n)^\frac12,  & p\ge4.
  \end{cases}
\end{align*}
\end{thm}
We postpone the proof of Theorem~\ref{thm:exp2} to Section~\ref{sec:exp2}.

\begin{rem} \label{rem:cf}
The rates in Theorem~\ref{thm:exp2} for semiflows are significantly stronger than those obtained in~\cite{LW24a} for maps due to a refinement\footnote{This refinement was suggested to us by Nicholas Fleming-V\'azquez.} of the strategy in Section~\ref{sec:exp2} below. It is easy to check that the same refinement leads to the improved rates also for maps.
\end{rem}

\begin{exam}[Intermittent semiflows]\label{ex:lsv}
We consider a class of intermittent semiflow $\Psi_t:M\to M$ with an intermittent Poincar\'e map, namely the LSV map $T:[0,1]\to [0,1]$ (see~\cite{LSV99}),
\begin{align*}
  T(x)=
  \begin{cases}
  x(1+2^\beta x^\beta) & x\in[0,\frac{1}{2}),\\
  2x-1  & x\in[\frac{1}{2},1],
  \end{cases}
\end{align*}
and H\"older return time function $h:[0,1]\to [1,\infty)$.
Here $\beta> 0$ is a parameter and the map $T$ and semiflow $\Psi_t$ are nonuniformly expanding of order $p$ for all $p<1/\beta$.

Paviato \cite{P24} derived the $1$-Wasserstein convergence rate $n^{-\frac{1}{6}+\delta}$
for $\beta\in (0,\frac{1}{3})$ and $n^{-\frac{1}{2}+\beta+\delta}$ for $\beta\in [\frac{1}{3},\frac{1}{2})$.
By Theorem~\ref{thm:exp2}, for $q< \frac{1}{2\beta}$, we obtain the $q$-Wasserstein convergence rate $n^{-\frac14}(\log n)^\frac12$
for $\beta\in(0,\frac{1}{4})$ and $n^{-\frac12+\beta+\delta}$ for $\beta\in [\frac{1}{4},\frac{1}{2})$, which improves Paviato's result in the range $\beta\in(0,\frac{1}{3})$.
\end{exam}

\section{Martingale-coboundary decompositions for semiflows}
\label{sec:MCD}

\setcounter{equation}{0}
In this section, we summarize some results for the suspension semiflow $F_t:Y^\varphi\to Y^\varphi$ defined in Section~\ref{sec:main}.
In Subsection~\ref{sec:GM}, we recall the notion of Gibbs-Markov semiflow.
In Subsection~\ref{sec:GMMCD}, we recall the martingale-coboundary decomposition for such semiflows from \cite{P24}, which extended the approach of \cite{KKM18} to continuous-time systems.

\subsection{Gibbs-Markov semiflows}
\label{sec:GM}

Let $F_t:Y^\varphi\to Y^\varphi$ be the suspension semiflow over the induced Gibbs-Markov map $F:Y\to Y$ with induced roof function $\varphi\in L^p(Y)$
and ergodic invariant probability measure $\mu^\varphi$
as in Subsection~\ref{sec:NUEsemi}.
By \cite[Proposition~4.1]{P24}, there is $C>0$ such that
\[
|\varphi(y)-\varphi(y')|\le C(\infYj\varphi)d(Fy,Fy')^\eta
\quad\text{for all $y,y'\in Y_j$, $j\ge1$}.
\]
Following~\cite{BBM19,P24}, we call $F_t:Y^{\varphi}\to Y^{\varphi}$ a {\em Gibbs-Markov semiflow of order $p$}.

For $j\ge 1$, set $Y_j^{\varphi}=\{(y,u)\in Y^{\varphi}:y\in Y_j\}$.
Given $w:Y^{\varphi}\to \R$ and $\eta\in (0,1]$, we define
\[|w|_\infty=\sup_{(y,u)\in Y^{\varphi}}|w(y,u)|, \quad
|w|_{\cV^{\eta}}=\sup_{j\ge 1}\sup_{(x,u),(y,u)\in Y_{j}^{\varphi}\atop x\neq y}
\frac{|w(x,u)-w(y,u)|}{(\infYj\varphi)d(Fx,Fy)^{\eta}}.
\]
We write $w\in \cV^{\eta}(Y^{\varphi})$ if $\|w\|_{\cV^{\eta}}=|w|_\infty+|w|_{\cV^{\eta}}<\infty$.
Let $\cV_0^{\eta}(Y^{\varphi})=\{w\in \cV^{\eta}(Y^{\varphi}): \int_{Y^{\varphi}} w \,\rmd \mu^\varphi=0\}$.

\subsection{Martingale-coboundary decompositions for Gibbs-Markov semiflows}
\label{sec:GMMCD}

Suppose that $F_t:Y^\varphi\to Y^\varphi$ is a Gibbs-Markov semiflow of order $p\ge2$.
Let $U_1w=w\circ F_1$ be the Koopman operator for the time-one map $F_1$ and $L_1$ the transfer operator of $F_1$, so
$\int L_1v\,w\,\rmd \mu^\varphi=\int v\,U_1w\,\rmd \mu^\varphi$ for
$v\in L^1(Y^{\varphi})$, $w\in L^\infty(Y^{\varphi})$.

\begin{prop} \label{prop:MCD}
Given $w\in \cV_0^\eta(Y^{\varphi})$, define a function $\psi:Y^{\varphi}\to \R$, $\psi=\int_0^1 w\circ F_s \,\rmd s$. Then
there exist $m\in L^p(Y^{\varphi})$, $\chi\in L^{p-1}(Y^{\varphi})$ such that $\psi=m+\chi\circ F_1-\chi$ and
$m\in \ker L_1$.

Moreover, there exists a constant $C>0$ such that for all $w\in \cV_0^\eta(Y^{\varphi})$,
\[
\|m\|_{L^p}\le C\|w\|_{\cV^\eta}, \quad \big\|\max_{1\le k\le n}|\chi\circ F_k-\chi|\,\big\|_{L^p}\le C\|w\|_{\cV^\eta}n^{1/p}.
\]
\end{prop}
\begin{proof}
This proposition is a summary of parts of~\cite[Propositions~4.6, 4.8 and 4.9]{P24}.
\end{proof}

For $1\le j\le n$, define the $\sigma$-algebra $\cG_{n,j}=F_{n-j}^{-1}\cB$.
where $\cB$ is the underlying $\sigma$-algebra on the space $Y^{\varphi}$.
\begin{prop} \label{prop:MDS}
Fix $n\ge 1$.
Then $\{m\circ F_{n-j}, \cG_{n,j}; 1\le j\le n\}$ is a sequence of martingale differences.
That is $\cG_{n,j-1}\subset \cG_{n,j}$,
$m\circ F_{n-j}$ is $\cG_{n,j}$-measurable,
and $\E(m\circ F_{n-j}|\cG_{n,j-1})=0$
for $0\le j\le n-1$.
\end{prop}

\begin{proof}
This is a standard consequence of the facts that
$F_1^{-1}\cB\subset \cB$,
$U_1L_1= \E(\cdot|F_1^{-1}\cB)$ and $m\in\ker L_1$
(see for example~\cite[Proposition~2.9]{KKM18}).
\end{proof}

Following~\cite[Equation~(4.15)]{P24},
define  $\breve w=(U_1L_1m^2)-\sigma^2= \E(m^2-\sigma^2|F_1^{-1}\cB)$.
Since $\int_{Y^{\varphi}} U_1L_1m^2 \,\rmd \mu^{\varphi}=\int_{Y^{\varphi}} L_1m^2 \,\rmd \mu^{\varphi}=\sigma^2$,
we have $\int_{Y^{\varphi}} \breve w \,\rmd \mu^{\varphi}=0$.

Define for $q>0$,
\[
\omega_q:[0,\infty)\to[0,\infty), \qquad \omega_q(t)=(t\ell(t))^q, \qquad
\ell(t)=\begin{cases} -\log t, & 0<t\le \tfrac13 \\ \quad\log 3, & t\ge \tfrac13 \end{cases}.
\]
Set $\omega_q(0)=0$.
Then $\ell$ and $\omega_q$ are continuous functions, with $\ell$ decreasing and $\omega_q$ increasing.
We now list some elementary properties of $\omega_q$.

\begin{prop} \label{prop:omega}
\begin{itemize}
\item[(a)]
$\omega_1(s+t)\le \omega_1(s)+\omega_1(t)$ for $s,t\ge0$.
\item[(b)]
$\|\,\omega_1(|Z|)\|_{L^r}\le 2\omega_1( \|Z\|_{L^r})$
for all $Z\in L^r$, $r\ge1$.
\end{itemize}
\end{prop}

\begin{proof}
(a) Since $\ell$ is decreasing, we have
$\omega_1(s+t)=s\ell(s+t)+t\ell(s+t)\le s\ell(s)+t\ell(t)= \omega_1(s)+\omega_1(t)$.

\vspace{1ex} \noindent
(b) If $|Z|\le \|Z\|_{L^r}$ then $\omega_1(|Z|)\le \omega_1(\|Z\|_{L^r})=\|Z\|_{L^r}\ell( \|Z\|_{L^r})$.
If $|Z|\ge \|Z\|_{L^r}$, then $\omega_1(|Z|)\le |Z|\ell( \|Z\|_{L^r})$.
Hence $\omega_1(|Z|)\le (\|Z\|_{L^r}+|Z|)\ell(\|Z\|_{L^r})$ and it follows that
$\|\,\omega_1(|Z|)\|_{L^r}\le 2\|Z\|_{L^r}\ell( \|Z\|_{L^r})
=2\omega_1(\|Z\|_{L^r})$.
\end{proof}

\begin{prop}\label{prop:sme}
Assume that $p>2$.
Let $w\in \cV_0^\eta(Y^{\varphi})$ and define
$\breve w_n=\sum_{i=0}^{n-1}\breve w\circ F_i$.
Then there exists a constant $C>0$ such that for all $n\ge 1$,
\[
\big\|\max_{1\le k\le n}n^{-1}|\breve w_k|\big\|_{L^{2(p-1)}}\le Cn^{-\frac12}
\quad\text{and}\quad
\big\|\max_{1\le k\le n}\omega_1(n^{-1}|\breve w_k|)\big\|_{L^{2(p-1)}}\le Cn^{-\frac12}\log n.
\]
\end{prop}

\begin{proof}
See~\cite[Corollary~4.18]{P24} for the first statement.
By Proposition~\ref{prop:omega}(b),
\(
\big\|\max_{1\le k\le n}\omega_1(n^{-1}|\breve w_k|)\big\|_{L^{2(p-1)}}\le
Cn^{-\frac12}\log n.
\)
\end{proof}

\section{Proof of Theorem~\ref{thm:exp2}}
\label{sec:exp2}

In this section, we prove Theorem~\ref{thm:exp2}. Recall that $\Psi_t:M\to M$ is a nonuniformly expanding semiflow of order $p>2$ with ergodic probability measure $\mu_M$.
Let $F_t:Y^{\varphi}\to Y^{\varphi}$ be the corresponding Gibbs-Markov semiflow with ergodic probability measure $\mu^\varphi$.
We have the measure-preserving semiconjugacy $\pi_M:Y^{\varphi}\to M$.

Let $v\in C_0^\eta(M)$.
By \cite[Proposition~4.3]{P24}, the lifted observable $w=v\circ \pi_M$ lies in $\cV_0^{\eta^2}(Y^{\varphi})$.

Define continuous processes $\tW_n\in C[0,1]$, $n\ge1$,  on $(Y^{\varphi}, \mu^\varphi)$ as
\[
\tW_n(t):=\frac{1}{\sqrt n}\int_{0}^{nt}w\circ F_s \,\rmd s, \quad t\in [0,1].
\]
Then $\tW_n=W_n\circ \pi_M=_d W_n$. So it follows from Proposition~\ref{prop:summ}(b) that $\tW_n\to _w W$, where $W$ is a Brownian  motion with variance $\sigma^2=\int_{Y^{\varphi}} m^2 \,\rmd \mu^\varphi$.
Moreover, $\cW_{\frac{p}{2}}(W_n, W)=\cW_{\frac{p}{2}}(\tW_n, W)$, so
we reduce to proving the following:

\begin{lem} \label{lem:exp}
Suppose that $F_t:Y^\varphi\to Y^\varphi$ is a Gibbs-Markov semiflow of order $p>2$ and let $\eta\in(0,1]$.
Suppose that $w\in \cV_0^{\eta}(Y^{\varphi})$.
Then there is a constant $C>0$ such that
\begin{align*}
  \cW_{\frac{p}{2}}(\tW_n,W)\le
  \begin{cases}
Cn^{-\frac12+\frac{1}{p}}(\log n)^\frac12, & 2<p<4,\\
  Cn^{-\frac14}(\log n)^\frac12,  & p\le 4.
\end{cases}
\end{align*}
\end{lem}

Let $\{m\circ F_{n-j},\cG_{n,j};1\le j\le n\}$ be the sequence of martingale differences in Proposition~\ref{prop:MDS}.
We define
\[
\zeta_{n,j}:=\frac{1}{\sqrt{n}\sigma}m\circ F_{n-j},
\quad\hbox{for } 1\le j\le n,
\]
and define the conditional variances
\[
V_{n,k}:=\sum_{j=1}^{k}\E(\zeta_{n,j}^2|\cG_{n,j-1}),
\quad\hbox{for } 1\le k\le n.
\]

Define the stochastic process $X_n$ with sample paths in $C[0,1]$ by
\begin{equation}\label{eq:Xn}
X_n(t):=\sum_{j=1}^{[nt]}\zeta_{n,j}+(nt-[nt])\zeta_{n,[nt]+1}, \qquad~~~~ t\in [0,1].
\end{equation}

\begin{rem} \label{rem:cf2} The corresponding definition of $X_n$ in~\cite{LW24a} involved a random time which was required in~\cite{AM19} and~\cite{P24}. It turns out not to be needed for the approach in~\cite{LW24a} and the simplification leads to stronger rates.
\end{rem}

\medskip

We recall the following standard argument from probability theory.
\begin{prop}\label{prop:mpe}
There is a constant $C>0$ such that
\(
\big\|\max_{1\le j\le n}| \zeta_{n,j}|\,\big\|_{L^p}\le Cn^{-\frac12+\frac{1}{p}}
\)
for all $n\ge1$.
\end{prop}

\begin{proof}
For all $n\ge 1$, we have
\(
\max_{0\le j\le n-1}|m\circ F_j|^p\le
\sum_{j=0}^{n-1}|m\circ F_j|^p
\)
so
$\|\max_{0\le j\le n-1}|m\circ F_j|\big\|_p\le n^{1/p}\|m\|_p$.
The result follows by definition of $\zeta_{n,j}$ and Proposition~\ref{prop:MCD}.
\end{proof}

\begin{prop}\label{prop:knt}
There exists a constant $C>0$ such that for all $n\ge 1$,
\[
\Big\|\max_{1\le k\le n}|V_{n,k}-\frac{k}{n}|\Big\|_{L^{2(p-1)}}\le Cn^{-1/2}
\quad\text{and}\quad
\Big\|\max_{1\le k\le n}\omega_1\big(\big|V_{n,k}-\frac{k}{n}\big|\big)\Big\|_{L^{2(p-1)}}\le Cn^{-\frac12}\log n.
\]
\end{prop}

\begin{proof}
We follow the argument of \cite[Proposition~5.5]{P24}, which is based on~\cite[Proposition~4.1]{AM19}.
Write
\begin{align*}
|V_{n,k}-\frac{k}{n}|&=\Big|\frac{1}{n\sigma^2}\sum_{j=1}^k\E(m^2\circ F_{n-j}|F_{n-(j-1)}^{-1}\cB)-\frac{k}{n}\Big|\\
&=\frac{1}{n\sigma^2}\Big|\sum_{j=1}^k\E(m^2-\sigma^2|F_{1}^{-1}\cB)\circ F_{n-j}\Big|
=\frac{1}{n\sigma^2}\Big|\sum_{j=1}^k \breve w\circ F_{n-j}\Big|.
\end{align*}
Hence the result follows from Proposition~\ref{prop:sme}.
\end{proof}

\begin{lem}\label{lem:xnw}
Let $B$ denote standard Brownian motion.
There exists a constant $C>0$ such that for all $n\geq 1$,
\[
\cW_{\frac{p}{2}}(X_n,B)\leq
\begin{cases}
Cn^{-\frac12+\frac{1}{p}}(\log n)^\frac12, & 2<p<4,\\
  Cn^{-\frac14}(\log n)^\frac12,  & p\ge4.
\end{cases}
\]
\end{lem}

\begin{proof}
The proof is similar to that in \cite[Lemma~4.4]{LW24a} and we follow the same steps omitting some arguments that are identical.

\vspace{1ex}
\noindent \textbf{(1)}
\hspace{1em}
By the Skorokhod embedding theorem (see \cite[Theorem A.1]{HH80}),
there exists a probability space (depending on $n$) supporting a
copy of the sequence $\{\zeta_{n,j},\,1\le j\le n\}$ (with unchanged joint distributions), a standard Brownian motion $B$, a sequence of nonnegative random variables $\tau_1,\ldots, \tau_n$ with $T_i=\sum_{j=1}^{i}\tau_j$, and a sequence of $\sigma$-fields $\cF_{i}$
generated by all events up to $T_i$ for $1\le i\le n$, such that for all $1\le i\le n$,
\begin{itemize}
\item $\sum_{j=1}^{i}\zeta_{n,j}=B(T_i)$;
\item $\E(\tau_{i}|\cF_{i-1})=\E(|\zeta_{n,i}|^{2}|\cG_{n,i-1})$ a.s.;
\item for any $p\ge 1$, there exists a constant $C_p>0$ such that
\[
\E(\tau_{i}^p|\cF_{i-1})\leq C_p\E(|\zeta_{n,i}|^{2p}|\cG_{n,i-1}) \quad \hbox{a.s.}
\]
\end{itemize}
On this probability space, we show that there exists a constant $C>0$ (independent of $n$) such that
\begin{align*}
\Big\|\sup_{t\in[0,1]}|X_n(t)-B(t)|\Big\|_{L^{\frac{p}{2}}}\leq
\begin{cases}
Cn^{-\frac12+\frac{1}{p}}(\log n)^\frac12, & 2<p<4,\\
  Cn^{-\frac14}(\log n)^\frac12,  & p\ge4.
  \end{cases}
\end{align*}
Then the result follows from the definition of the Wasserstein distance.

For ease of exposition, we write $\zeta_j$ and $V_k$ instead of $\zeta_{n,j}$ and $V_{n,k}$ respectively. By~\eqref{eq:Xn},
\begin{align}\label{eq:set}
X_n(t)=B(T_{[nt]})+(nt-[nt])\big(B(T_{[nt]+1})-B(T_{[nt]})\big),\quad \hbox{for}~t\in [0,1].
\end{align}

\vspace{1ex}
\noindent \textbf{(2)}
\hspace{1em}
Following the argument in \cite[Lemma~4.4]{LW24a},
$T_k-V_k=\sum_{j=1}^k (\tau_j-\E(\tau_j|\cF_{j-1}))$  is an $L^{\frac{p}{2}}$ martingale so it follows by Burkholder's inequality that
\[
\Big\|\max_{1\le k\le n}|T_k-V_{k}|\Big\|_{L^{\frac{p}{2}}} \le
\begin{cases}
  Cn^{\frac{2}{p}}
\max_{1\le k\le n}\|\tau_k\|_{L^{\frac{p}{2}}}
, & 2<p<4,\\
  Cn^{\frac{1}{2}}
\max_{1\le k\le n}\|\tau_k\|_{L^{\frac{p}{2}}}
,  & p\ge4.
  \end{cases}
\]
But
$\|\tau_k\|_{L^{\frac{p}{2}}} \le C\|\zeta_k\|_{L^p}^2=C\sigma^{-2}n^{-1}\|m\|_{L^p}^2\le C n^{-1}$. Hence
\[
\Big\|\max_{1\le k\le n}|T_k-V_{k}|\Big\|_{L^{\frac{p}{2}}} \le
\begin{cases}
  Cn^{-1+\frac{2}{p}}
, & 2<p<4,\\
  Cn^{-\frac{1}{2}}
,  & p\ge4.
  \end{cases}
\]
Similarly, by Proposition~\ref{prop:omega}(b),
\[
\big\|\max_{1\le k\le n} \omega_1(|T_k-V_k|)\big\|_{L^{\frac{p}{2}}}\le
\begin{cases}
Cn^{-1+\frac{2}{p}}\log n,  & 2<p<4,\\
C n^{-\frac12}\log n, & p\ge4.
\end{cases}
\]

In addition, by Proposition~\ref{prop:knt},
\[
\|V_n-1\|_{L^{2(p-1)}}\le  C n^{-\frac{1}{2}} \quad\text{and}\quad
\big\|\,\omega_1(|V_n-1|)\big\|_{L^{2(p-1)}}\le  C n^{-\frac12}\log n.
\]

\vspace{1ex}
\noindent \textbf{(3)}
\hspace{1em}
We now estimate $|X_n-B|$ on the set $\{|T_n-1|> 1\}$.
By Chebyshev's inequality and the estimates in \textbf{(2)},
\begin{equation*}
\mu^\varphi(|T_n-1|>1)\le \E|T_n-1|^{\frac{p}{2}} \le
\begin{cases}
  Cn^{-\frac{p}{2}+1}, & 2<p<4,\\
  Cn^{-\frac{p}{4}},  & p\ge4.
  \end{cases}
\end{equation*}
Hence we deduce that
\begin{align*}
& \Big\| 1_{\{|T_n-1|>1\}}\sup_{t\in[0,1]}|X_n(t)-B(t)|\Big\|_{L^{\frac{p}{2}}}\\
& \qquad \le \big(\mu^\varphi(|T_n-1|>1)\big)^{\frac{1}{p}}\Big(\big\|\sup_{t\in[0,1]}|X_n(t)|\big\|_{L^p}+\big\|\sup_{t\in[0,1]}|B(t)|\big\|_{L^p}\Big)\\
&\qquad \le
\begin{cases}
  Cn^{-\frac{1}{2}+\frac{1}{p}}, & 2<p<4,\\
  Cn^{-\frac{1}{4}},  & p\ge4.
\end{cases}
\end{align*}

\noindent \textbf{(4)}
\hspace{1em}
By \textbf{(3)}, it remains to estimate $|X_n-B|$ on the set $\{|T_n-1|\le 1\}$. Now
\[
\Big\| 1_{\{|T_n-1|\le 1\}}\sup_{t\in[0,1]}|X_n(t)-B(t)|\Big\|_{L^{\frac{p}{2}}} \le I_1+I_2,
\]
where
\[
I_1  =:
\Big\| \sup_{t\in[0,1]}|X_n(t)-B(T_{[nt]})|\Big\|_{L^{\frac{p}{2}}}, \quad
I_2  =:\Big\|1_{\{|T_n-1|\le 1\}}\sup_{t\in[0,1]}|B(T_{[nt]})-B(t)|\Big\|_{L^{\frac{p}{2}}}.
\]
By \eqref{eq:set} and Proposition~\ref{prop:mpe},
\[
I_1
\le \big\| \max_{1\le k\le n}|\zeta_{k}|\big\|_{L^p}
\le Cn^{-\frac{1}{2}+\frac{1}{p}}.
\]

\noindent \textbf{(5)}
\hspace{1em}
By \textbf{(3)} and \textbf{(4)}, it remains to estimate $I_2$.
We claim that
\begin{equation}\label{eq:te}
\big\|\sup_{t\in[0,1]}\omega_{\frac12}(|T_{[nt]}-t|)\big\|_{L^p}
\le
\begin{cases}
Cn^{-\frac12+\frac{1}{p}}(\log n)^\frac12, & 2<p<4,\\
  Cn^{-\frac14}(\log n)^\frac12,  & p\ge 4.
\end{cases}
\end{equation}

Assuming the claim, we proceed as in~\cite[Lemma~4.4]{LW24a}, but with a slight improvement.
By Theorem~\ref{thm:holder},
\begin{equation}\label{eq:holder}
\Big\|\sup_{0\le s<t\le2}\frac{|B(t)-B(s)|}{\omega_{\frac12}(t-s)}\Big\|_{L^p}< \infty.
\end{equation}
On the set $\{|T_n-1|\le 1\}$, note that
\[
\sup_{t\in[0,1]}|B(T_{[nt]})-B(t)|\le \Big(\sup_{0\le s<t\le 2}\frac{|B(t)-B(s)|}{\omega_{\frac12}(t-s)}\Big)\big(\sup_{t\in[0,1]}\omega_{\frac12}(|T_{[nt]}-t|)\big).
\]
By H\"{o}lder's inequality,
\begin{align*}
I_2
\le &\Big\|\sup_{0\le s<t\le 2}\frac{|B(t)-B(s)|}{\omega_{\frac12}(t-s)}
\Big\|_{L^p}\big\|\sup_{t\in[0,1]}\omega_{\frac12}(|T_{[nt]}-t|)\big\|_{L^p}
\le
C\big\|\sup_{t\in[0,1]}\omega_{\frac12}(|T_{[nt]}-t|)\big\|_{L^p},
\end{align*}
so the result follows from the claim.

It remains to verify the claim. Given $t\in(0,1]$, choose $k$ such that
$t\in(\frac{k-1}{n},\frac{k}{n}]$. Then
$T_{[nt]}-t=(T_k-V_k)+(V_k-\frac{k}{n})+(\frac{k}{n}-t)$ and so
\[
\sup_{t\in[0,1]}|T_{[nt]}-t|\le \max_{1\le k\le n}|T_k-V_k|+\max_{1\le k\le n}|V_k-\tfrac{k}{n}|+\tfrac{1}{n}.
\]
Applying the estimates for
$\big\|\max_{1\le k\le n} \omega_1(|T_k-V_k|)\big\|_{L^{\frac{p}{2}}}$ from
\textbf{(2)} and
$\big\|\max_{1\le k\le n}\omega_1(|V_k-\frac{k}{n}|)\big\|_{L^p}$ in Proposition~\ref{prop:knt}
and using Proposition~\ref{prop:omega}(a),
\[
\big\|\sup_{t\in[0,1]}\omega_1(|T_{[nt]}-t|)\big\|_{L^\frac{p}{2}}
\le
\begin{cases}
Cn^{-1+\frac{2}{p}}\log n, & 2<p<4,\\
  Cn^{-\frac12}\log n,  & p\ge 4.
\end{cases}
\]
The claim follows since
$\big\|\sup_{t\in[0,1]}\omega_\frac12(|T_{[nt]}-t|)\big\|_{L^p}=
\big\|\sup_{t\in[0,1]}\omega_1(|T_{[nt]}-t|)\big\|_{L^\frac{p}{2}}^\frac12$.
\end{proof}

Define $g:C[0,1]\to C[0,1]$ by $g(u)(t):=u(1)-u(1-t)$.

\begin{lem}\label{lem:WnXn}
Let $p>2$. Then there exists a constant $C>0$ such that $\cW_p(g\circ \tW_n,\sigma X_n)\leq Cn^{-\frac{1}{2}+\frac{1}{p}}$ for all $n\geq 1$.
\end{lem}

\begin{proof}
The proof follows that in \cite[Lemma~4.7]{AM19} with obvious notional changes.
Write
\begin{align*}
g( \tW_n(t))-\sigma X_n(t)
& =\frac{1}{\sqrt n}\Big(\int_{n-[nt]}^nw\circ F_s \,\rmd s-\sum_{j=1}^{[nt]}m\circ F_{n-j}\Big)+E_n(t)
\\ & =\frac{1}{\sqrt n}\Big(\sum_{j=1}^{[nt]}\psi\circ F_{n-j}-\sum_{j=1}^{[nt]}m\circ F_{n-j} \Big)+E_n(t)
\\ & =\frac{1}{\sqrt n}(\chi\circ F_n-\chi\circ F_{n-[nt]})+E_n(t),
\end{align*}
where
\(
|E_n(t)|\le \frac{1}{\sqrt n}\|w\|_\infty + \max_{0\le j\le n}|\zeta_{n,j}|,
\)
so
$\big\|\sup_{t\in[0,1]}|E_n(t)|\big\|_{L^p}\le Cn^{-\frac{1}{2}+\frac{1}{p}}$
by Proposition~\ref{prop:mpe}.

By Proposition~\ref{prop:MCD},
\[
\big\|\sup_{t\in[0,1]}|\chi\circ F_n-\chi\circ F_{n-[nt]}|\big\|_p
\le 2 \big\|\max_{0\le j\le n}|\chi\circ F_j-\chi|\big\|_p \le n^{1/p}.
\]
Hence
$\big\|\sup_{t\in[0,1]}|g(\tW_n(t))-\sigma X_n(t)|\big\|_{L^p}\le
Cn^{-\frac{1}{2}+\frac{1}{p}}$
and the result follows.
\end{proof}

\begin{proof}[Proof of Lemma~\ref{lem:exp}]
Note that $g\circ g= Id$ and $g$ is Lipschitz with $\Lip g$ $\le 2$.
Also, $g(W)=_d W=_d\sigma B$.
By the Lipschitz mapping theorem (see~\cite[Proposition~2.4]{LW24a}),
\[
\cW_{\frac{p}{2}}(\tW_n,W)= \cW_{\frac{p}{2}}(g(g\circ \tW_n),g(g\circ W))\le
2\cW_{\frac{p}{2}}(g\circ \tW_n,g\circ W)
=2\cW_{\frac{p}{2}}(g\circ \tW_n,\sigma B).
\]
By Lemmas~\ref{lem:xnw} and~\ref{lem:WnXn},
\begin{align*}
\cW_{\frac{p}{2}}(g\circ \tW_n, \sigma B)
&\le \cW_p(g\circ \tW_n, \sigma X_n)+\cW_{\frac{p}{2}}(\sigma X_n,\sigma B)
\le
\begin{cases}
Cn^{-\frac12+\frac{1}{p}}(\log n)^\frac12, & 2<p<4,\\
  Cn^{-\frac14}(\log n)^\frac12,  & p\ge4.
\end{cases}
\end{align*}
The result follows.

\end{proof}

\section{Nonuniformly hyperbolic flows}
\label{sec:NUHflow}

\setcounter{equation}{0}

In this section, we show how the main results in Section~\ref{sec:main} extend from nonuniformly expanding semiflows to nonuniformly hyperbolic flows.
The techniques for quotienting out stable directions are more-or-less standard following~\cite{B75,R78,S72} but remain somewhat tricky for flows. For Axiom~A flows, this was done in~\cite{MT02}. Then~\cite{AMV15} covered the case of singular hyperbolic flows, including the classical Lorenz attractor. In the more general situation considered here, we follow
arguments in \cite[Section~6]{P24}.  We note that these techniques are restricted to the case where there is exponential contraction along stable manifolds, even though we allow nonuniform expansion.\footnote{The condition of exponential contraction along stable manifolds is relaxed somewhat in~\cite{BBM19} in the study of rates of decay of correlations for flows.}

In Subsection~\ref{sec:setup}, we introduce the
setup. In Subsection~\ref{sec:mainflows}, we state our main results for flows and give several examples. In Subsection~\ref{sec:hyp2}, we provide a sketch of the main result for flows.

\subsection{The setup}
\label{sec:setup}

Let $(M,d)$ be a bounded metric space and let $\Psi_t:M\to M$ be a flow, satisfying $\Psi_0=\text{Id}$ and $\Psi_{t+s}=\Psi_t\circ\Psi_s$
for $s,t\in \R$.
As in Section~\ref{sec:main}, we assume continuous dependence on initial conditions \eqref{eq:con} and Lipschitz continuity in time \eqref{eq:tim}.
We suppose that there is a Borel subset $X\subset M$ with first return time $h:X\to\R^+$ satisfying $h\in C^\eta(X)$ and $\inf h\ge1$.

$\bullet$  We suppose that there is a ``uniformly hyperbolic'' subset $Y\subset X$, an at most countable measurable partition $\{Y_j\}$ of $Y$, and an integrable return time function $r:Y\to \Z^+$ constant on each $Y_j$ such that $T^{r(y)}y\in Y$. Define $F:Y\to Y$ as $Fy=T^{r(y)}y$.
We suppose that $\mu_Y$ is an ergodic $F$-invariant Borel probability measure on $Y$.
We define a separation time $s(y,y')$ on $Y$ as the least integer $n\ge0$ such that $F^ny$ and $F^ny'$ belong to different partition elements.

$\bullet$ We suppose that there is a measurable partition $\cQ^s$ of $Y$ consisting of ``stable leaves'' refining $\{Y_j\}$.\footnote{More standard notation is $\cW^s$ but we already used $W$ for processes and $\cW$ for Wasserstein distance.} Let $Q^s(y)$ be the stable leaf containing $y\in Y$. We assume that $FQ^s(y)\subset Q^s(Fy)$.
Define the quotient space $\bY=Y/\cQ^s$ with projection $\bpi:Y\to\bY$
and partition $\{\bY\!_j\}$ where $\bY\!_j=\bpi Y_j$.
We also have the quotient map $\bF:\bY\to\bY$ with ergodic invariant probability measure $\bmu_Y=\bpi_*\mu_Y$.  We suppose that $\bF$ is a Gibbs-Markov map as in Section~\ref{sec:NUE}.

$\bullet$ We require that there is a measurable subset $\tY\subset Y$ such that for every $y\in Y$, there is a unique $\tilde y\in \tY\cap Q^s(y)$. Let $\pi:Y\to \tY$ denote the associated projection.
Let $\beta_n(y)=N$ be the unique integer such that
\[
\sum_{j=0}^{N-1}r(F^jy)\le n< \sum_{j=0}^{N}r(F^jy).
\]
This counts the number of ``good'' returns of the map $T$ to $Y$ by time $n$.
We suppose that there exist $C>0$ and $\gamma\in(0,1)$ such that
\begin{equation} \label{eq:T}
d(T^ny,T^ny')\le C(\gamma^n d(y,y')+\gamma^{s(y,y')-\beta_n(y)})
\quad\text{for all $y,y'\in Y$, $n\ge0$,}
\end{equation}
and
\begin{equation} \label{eq:T2}
d(T^ny,T^ny')\le C\gamma^{s(y,y')-\beta_n(y)}
\quad\text{for all $y,y'\in \tY$, $n\ge0$.}
\end{equation}

\begin{rem} \label{rmk:T}
In particular, we have contraction of $T$ along stable leaves:
\[
d(T^ny,T^ny')\le C\gamma^n d(y,y')
\quad\text{for all $y,y'\in Y$ with $y'\in Q^s(y)$, $n\ge0$.}
\]
\end{rem}

Let $\varphi:Y\to [1,\infty)$ be defined as $\varphi(y)=\sum_{i=0}^{r(y)-1}h(T^iy)$. Since $r\in L^1(Y)$ and $\varphi\le |h|_\infty r$,
we have $\varphi\in L^1(Y)$. Define the suspension $Y^{\varphi}=\{(y,u)\in Y\times \R: 0\le u\le \varphi(y)\}/\sim$,
where $(y,\varphi(y))\sim (Fy,0)$. The suspension flow $F_t:Y^{\varphi}\to Y^{\varphi}$ is given by $F_t(y,u)=(y,u+t)$ computed modulo identifications.
The projection $\pi_M:Y^{\varphi}\to M$, $\pi_M(y,u)=\Psi_uy$ is a semiconjugacy between $F_t$ and $\Psi_t$.
We have an ergodic $F_t$-invariant probability measure $\mu^\varphi=(\mu_Y\times \text{Lebesgue})/\bar \varphi$,
where $\bar\varphi=\int_Y\varphi\,\rmd \mu_Y$.
Then $\mu_M=(\pi_M)_*\mu^\varphi$ is an ergodic $\Psi_t$-invariant probability measure on $M$.

A flow $(\Psi_t,M,\mu_M)$ satisfying these assumptions is called a
\emph{nonuniformly hyperbolic flow of order $p$} if $r$ (and hence $\varphi$) is $L^p$.

\subsection{Main results for flows}
\label{sec:mainflows}
Let $v\in C_0^{\eta}(M)$ and define continuous processes $W_n\in C[0,1]$ as
\[
W_n(t)=\frac{1}{\sqrt{n}}\int_0^{nt}v\circ\Psi_s\,\rmd s,\quad t\in[0,1].
\]
Then Proposition~\ref{prop:summ} still holds for nonuniformly hyperbolic flows; see~\cite{KKM18,KKM22,MN05,MT12} for example. The statements of our main results are completely analogous to those of Theorems~\ref{thm:exp1} and~\ref{thm:exp2}.

\begin{thm}\label{thm:hyp1}
Let $\Psi_t:M\to M$ be a nonuniformly hyperbolic flow of order $p>2$ and $v\in C_0^{\eta}(M)$. Then $\cW_q(W_n,W)\to 0$ in $C[0,1]$ for all $1\le q< 2(p-1)$.
\end{thm}

\begin{proof}
Since $W_n$ has a finite moment of order $2(p-1)$, together with Proposition~\ref{prop:lid} and the fact that $W_n\to_{w} W$ as $n\to\infty$, we can obtain the conclusion just as in the proof of Theorem~\ref{thm:exp1}.
\end{proof}

\begin{thm}\label{thm:hyp2}
Let $\Psi_t:M\to M$ be a nonuniformly hyperbolic flow of order $p>2$ and $v\in C_0^{\eta}(M)$.  Then there is a constant $C>0$ such that
\begin{align*}
  \cW_{\frac{p}{2}}(W_n,W)\le
  \begin{cases}
Cn^{-\frac12+\frac{1}{p}}(\log n)^\frac12, & 2<p<4,\\
  Cn^{-\frac14}(\log n)^\frac12,  & p\ge4.
  \end{cases}.
\end{align*}
\end{thm}

\begin{exam}[Axiom A flows] \label{ex:solenoid}
By~\cite{B73}, Axiom A flows~\cite{BR75,S67} are nonuniformly hyperbolic of order $p$ for all $1\le p<\infty$ (indeed $r$ and hence $\varphi$ are bounded), so Theorem~\ref{thm:hyp2} applies to $\cW_q$ for all $q\ge1$ and we obtain the rate
$\cW_q(W_n,W)\leq Cn^{-\frac{1}{4}}(\log n)^\frac12$.
This includes Anosov flows~\cite{A67} (such as geodesic flows on negatively curved manifolds) and solenoids.
\end{exam}

\begin{exam}[Planar periodic Lorentz gases]
The $2$-dimensional periodic Lorentz gas is a model of electron gases in metals studied by Sina{\u\i}~\cite{S70}. The Lorentz flow is a dispersing billiard flow on $M=(\T^2\setminus\Omega)\times S^1$, $\Omega=\bigcup_{i=1}^{k}\Omega_i$, where the obstacles $\Omega_i$ are disjoint convex regions with $C^3$ boundaries of nonvanishing curvature. The Poincar\'{e} map (or collision map) $T$ is a dispersing billiard defined on $X=\partial \Omega\times[-\frac{\pi}{2},\frac{\pi}{2}]$. Under the finite horizon condition, which means that the roof function (or collision time) $h$ is bounded, Young~\cite{Y98} demonstrated that $T$ has exponential decay of correlations. In particular, it follows from~\cite{Y98} that $\Psi_t$ is uniformly hyperbolic of order $p$ for all $1\le p<\infty$.
By Theorem~\ref{thm:hyp2}, we obtain  $\cW_q(W_n,W)\leq Cn^{-\frac{1}{4}}(\log n)^\frac12$ for all $q\ge1$.
\end{exam}

\begin{exam}[Dispersing billiards with cusps]
Next, we consider dispersing billiard flows with cusps, where the boundary curves are all dispersing but the interior angles at corner points are zero. Chernov and Markarian~\cite{CM07} proved that the billiard map has slow decay of correlations with rate $1/n$. However the collision time is not bounded below.
By considering an an
alternative cross section $X'$ bounded away from the cusps, B\'alint and Melbourne~\cite{BM08} proved that the corresponding Lorentz flow is superpolynomially mixing. As a byproduct of  proving this, they showed that the flow is nonuniformly hyperbolic of order $p$ for all $1\le p<\infty$. Hence, by Theorem~\ref{thm:hyp2}, we obtain  $\cW_q(W_n,W)\leq Cn^{-\frac{1}{4}}(\log n)^\frac12$ for all $q\ge1$.
\end{exam}

\begin{exam}[The Lorenz attractor]
Statistical limit laws and exponential decay of correlations were proved for the classical Lorenz attractor in~\cite{AM16,AMV15,BM18,HM07}. In particular, it follows from~\cite{HM07} that the Lorenz attractor is nonuniformly hyperbolic of order $p$ for all $1\le p<\infty$. Hence, by Theorem~\ref{thm:hyp2}, we obtain  $\cW_q(W_n,W)\leq Cn^{-\frac{1}{4}}(\log n)^\frac12$ for all $q\ge1$.
The same result holds for singular hyperbolic attractors by~\cite{ArM19}.
\end{exam}

\begin{exam}[Intermittent solenoidal flows]
Let $T_0:[0,1]\to[0,1]$ be the intermittent map considered in Example~\ref{ex:lsv}.
We consider a diffeomorphism $T:X\to X$ introduced in \cite{AP08}, obtained by replacing the expanding map in the classical solenoid map by $T_0$, with exponential contraction along stable leaves.
Hence we can construct an intermittent flow $\Psi_t:M\to M$ with $T:X\to X$ as a Poincar\'e map and H\"older return time function
$h:X\to [0,\infty)$.
This yields a uniformly expanding flow of order $p$ for all $p<1/\beta$ where
$\beta$ is the parameter in Example~\ref{ex:lsv}.
Hence by Theorem~\ref{thm:hyp2}, for $q< \frac{1}{2\beta}$, we obtain the $q$-Wasserstein convergence rate $n^{-\frac14}(\log n)^\frac12$
for $\beta\in(0,\frac{1}{4})$ and $n^{-\frac12+\beta+\delta}$ for $\beta\in [\frac{1}{4},\frac{1}{2})$.
\end{exam}

\subsection{Proof of Theorem~\ref{thm:hyp2}}
\label{sec:hyp2}

Following the arguments in \cite[Section~6]{P24}, the proof consists of two ingredients:

\begin{itemize}
\item[(1)] Reduction to a roof function that is constant along stable leaves.
\item[(2)] Reduction to an observable that is constant along stable leaves.
\end{itemize}
We sketch the constructions, referring to \cite{P24} for further details.

\vspace{1ex}
\noindent{\bf Step 1: Reduction to a constant roof function along stable leaves.}
Let $\gamma_1=\gamma^{\eta/2}$,
$\gamma_2=\gamma_1^\eta$.
Define
\[
\chi_Y:Y\to\R, \qquad
\chi_Y(y)=\sum_{n=0}^\infty  \big\{
(h\circ T^n) (\pi y)
- (h\circ T^n) (y) \big\}.
\]
By~\eqref{eq:T} and~\eqref{eq:T2},
$\chi_Y\in L^\infty$ and
\[
|\chi_Y(y)-\chi_Y(y')|\le C\big(d(y,y')^\eta+\gamma_1^{s(y,y')}\big)
\quad\text{
for all $y,y'\in Y$},
\]
(see~\cite[Lemma~8.4]{BBM19} or~\cite[Proposition~6.7]{P24}).

Now define
\[
\tilde \varphi=\varphi+\chi_Y\circ F-\chi_Y.
\]
Then $\tilde\varphi:Y\to\R$ lies in $L^p$ and, by construction,
$\tilde\varphi$ is constant along stable leaves.
Replacing $F$ and $\varphi$ by $F^k$ and $\sum_{j=0}^{k-1}\varphi\circ F^j$ for a fixed sufficiently large $k$, we can suppose without loss of generality that $\inf\tilde \varphi\ge1$ so that $\tilde \varphi$ is a roof function.
By~\cite[Proposition~6.1]{BBM19} or~\cite[Proposition~6.8]{P24},
\[
|\tilde \varphi(y)-\tilde \varphi(y')|\le C(\infYj \tilde \varphi)\gamma_1^{s(y,y')}
\quad\text{for all $y,y'\in Y_j$, $j\ge 1$.}
\]

Let $Y^{\tilde \varphi}$ be the suspension over the map $F:Y\to Y$ with roof function $\tilde \varphi$ and
let $\tF_t:Y^{\tilde \varphi}\to Y^{\tilde \varphi}$ be the associated suspension flow with ergodic invariant probability measure
$\mu_{\tilde \varphi}$.
We can define a measure-preserving semiconjugacy
\[
g:Y^{\tilde\varphi}\to Y^\varphi,
\qquad g(y,u)=(y,u+\chi_Y(y)),
\]
between $\tF_t$ and $F_t$.
Then $\tilde\pi_M=\pi_M\circ g:Y^{\tilde\varphi}\to M$
is a measure-preserving semiconjugacy between $\tF_t$ and $\Psi_t$.

By~\cite[Proposition~6.10]{P24}, there exists $\gamma_3\in (0,\gamma_2)$ such that
\begin{equation} \label{eq:expflow}
d\big(\tilde\pi_M\circ \tilde F_t(y,0),\tilde\pi_M\circ \tilde F_t(y',0)\big)\le C\gamma_3^t
\end{equation}
for all $y,y'\in Y$ with $y'\in Q^s(y)$ and all $t\ge0$.

Finally, let $v\in C^\eta(M)$. By~\cite[Proposition~6.11]{P24}, the lifted observable $v\circ\tilde\pi_M:Y^{\tilde\varphi}\to\R$
 satisfies
\[
|v\circ\tilde\pi_M(y,u)-v\circ\tilde\pi_M(y',s)|\le C\big((\infYj\tilde\varphi)\gamma_2^{s(y,y')}+|u-s|^\eta\big)
\]
for all $(y,u)$, $(y',s)\in Y^{\tilde\varphi}$ such that $y,y'\in \tY_j$.

\vspace{1ex}
\noindent{\bf Step 2: Reduction to a constant observable along stable leaves.}
By Step~1, we can assume without loss that the flow $\Psi_t$ is modelled by a suspension flow $F_t$ defined on a suspension $Y^\varphi$ with roof function $\varphi$ constant along stable leaves and satisfying
\[
|\varphi(y)-\varphi(y')|\le C(\infYj \varphi)\gamma_1^{s(y,y')}
\quad\text{for all $y,y'\in Y_j$, $j\ge 1$.}
\]
Moreover, we may assume the exponential contraction~\eqref{eq:expflow} along stable leaves and we may suppose that the lifted observable
$v\circ\pi_M:Y^\varphi\to\R$ satisfies
\[
|v\circ\pi_M(y,u)-v\circ\pi_M(y',s)|\le C\big((\infYj\varphi)\gamma_2^{s(y,y')}+|u-s|^\eta\big)
\]
for all $(y,u)$, $(y',s)\in Y^{\varphi}$ such that $y,y'\in \tY_j$.

Let $\bY^{\bar\varphi}$ be the suspension over the quotient map $\bF:\bY\to\bY$ with roof function $\bar\varphi$ and
let $\bF_t:\bY^{\bar\varphi}\to \bY^{\bar\varphi}$ be the associated suspension flow. Since $\bF$ is a Gibbs-Markov map and the roof function
$\bar\varphi$ satisfies
\[
|\bar\varphi(y)-\bar \varphi(y')|\le C(\infbYj \bar \varphi)\gamma_1^{s(y,y')}
\quad\text{for all $y,y'\in \bY\!_j$, $j\ge 1$,}
\]
the quotient semiflow $\bF_t:\bY^{\bar\varphi}\to \bY^{\bar\varphi}$ is a Gibbs-Markov semiflow
w.r.t.\ the metric $d(y,y')=\gamma_1^{s(y,y')}$,

Given $v\in C^\eta(M)$, define
\[
\chi:Y^\varphi\to\R, \qquad
\chi=\sum_{n=0}^\infty\{ v\circ\pi_M\circ F_n -
v\circ\pi_M\circ F_n\circ \pi\},
\]
where $\pi(y,u)=(\pi y,u)$.
It follows from~\eqref{eq:expflow} that $\chi\in L^\infty(Y^\varphi)$.

Define
\[
\hv:Y^\varphi\to\R, \qquad \hv=v\circ\pi_M-\chi+\chi\circ F_1.
\]
By construction, $\hv$ is constant along stable leaves and hence projects to an observable $\bv:\bY^{\bar\varphi}\to \R$.

By~\cite[Proposition~6.17]{P24}, there exists $\gamma_4\in(0,\gamma_3)$ such that
\[
|\hv(y,u)-\hv(y',u)|\le C(\infYj\varphi)\gamma_4^{s(y,y')}
\]
for all $(y,u),(y',u)\in Y^\varphi$ with $y,y'\in Y_j$, $j\ge1$.
Hence $\bv\in \cV^1(\bY^{\bar\varphi})$ with the metric $d(y,y')=\gamma_4^{s(y,y')}$.

\begin{proof}[Proof of Theorem~\ref{thm:hyp2}]
First we claim that to get the rates for $W_n$ it suffices to prove them for the sequence
\[
\bW\!_n(t)=\frac{1}{\sqrt n}\int_0^{nt}\bv\circ \bF_s\,\rmd s
\]
defined on the probability space $(\bY^{\bar \varphi}, \bar\mu^{\bar \varphi})$.

Consider the sequences
\[
W'_n(t)=\frac{1}{\sqrt n}\int_0^{nt}v\circ\pi_M\circ F_s\,\rmd s,\qquad W''_n(t)=\frac{1}{\sqrt n}\int_0^{nt}\hv\circ F_s\,\rmd s.
\]
Note that $W'_n=W_n\circ \pi_M$ and $W''_n=\bW\!_n\circ \bar\pi$. Since $\pi_M$ and $\bar\pi$ are
measure-preserving, $W_n=_dW'_n$ and $\bW\!_n=_dW''_n$. For all $t\in \R$,
\[
\int_0^{t}(\chi-\chi\circ F_1)\circ F_s\,\rmd s=\int_0^{t}\chi\circ F_s\,\rmd s-\int_1^{t+1}\chi\circ F_s\,\rmd s=
\int_0^{1}\chi\circ F_s\,\rmd s-
\int_t^{t+1}\chi\circ F_s\,\rmd s.
\]
It follows that for all $q\ge1$,
\begin{align*}
\cW_q(W_n,\bW\!_n)  =\cW_q(W'_n, W''_n) & \le \big|\sup_{t\in [0,1]}|W'_n(t)-W''_n(t)|\big|_\infty
\\ & =n^{-1/2}\Big|\sup_{t\in [0,1]}\Big|\int_0^{nt}(\chi-\chi\circ F_1)\circ F_s\,\rmd s\Big|\,\Big|_\infty\le 2n^{-1/2}|\chi|_\infty.
\end{align*}
proving the claim.

Since $\bF_t$ is a Gibbs-Markov semiflow and $\bv\in\cV^1_0(\bY^{\bar\varphi})$,
the rate for $\cW_{\frac{p}{2}}(\bW\!_n, W)$ follows from Lemma~\ref{lem:exp}.
\end{proof}

\appendix

\section{Kolmogorov-type estimate}

The following result is presumably known to experts. Since we could not find a statement in the literature, we provide a proof.\footnote{The proof is based on an argument in ``Moments of the H\"{o}lder norm of Brownian process'', M.\ Kwa\'{s}nicki, Math.\ Overflow, 2018. https://mathoverflow.net/questions/304238/moments-of-the-hölder-norm-of-brownian-process.}
Recall that $\omega=\omega_\frac12$ is an increasing function agreeing with
$(t|\log t|)^\frac12$ for $t$ small.

\begin{thm}  \label{thm:holder}
$\E\Big(\sup_{0\le s<t\le 1}\dfrac{|B(t)-B(s)|}{\omega(t-s)}\Big)^p<\infty$ for all $1\le p<\infty$.
\end{thm}

\begin{proof}
Let $S_{T,c}=
\sup_{0\le s<t\le T,\,t-s<c}\frac{|B(t)-B(s)|}{\omega(t-s)}$.
Since
$\E \sup_{0\le s<t\le 1}|B(t)-B(s)|^p
\le 2^p\E \sup_{0\le t\le 1}|B(t)|^p<\infty$,
it suffices to show that
$\E S_{1,c}^p<\infty$ for some $c=c(p)\in(0,1)$.

Let $\eps>0$.
We fix $c\in(0,1)$ so that $|\log t/\log 2t|\le 1+\eps$ for $0<t<c$.
Then $\omega(2t)\ge 2^{1/2}(1+\eps)^{-1/2}\omega(t)$ for $0<t<c$, and
\begin{align*}
S_{2,c} & =
\sup_{0\le s<t\le 2,\,t-s<c}\frac{|B(t)-B(s)|}{\omega(t-s)}
=\sup_{0\le s<t\le 1,\,t-s<c/2}\frac{|B(2t)-B(2s)|}{\omega(2(t-s))}
\\ & \le 2^{-1/2}(1+\eps)^{1/2}\sup_{0\le s<t\le 1,\,t-s<c}\frac{|B(2t)-B(2s)|}{\omega(t-s)}
=_d (1+\eps)^{1/2}S_{1,c}.
\end{align*}
Hence
\begin{equation} \label{eq:12}
\P\big(S_{1,c}>x\big)\ge \P(S_{2,c}>(1+\eps)^{1/2}x)
\quad\text{for all $x>0$}.
\end{equation}

Now $S_{1,c}\le S_{1,1}<\infty$ a.e.\ by~\cite[Exercise~2.4.8]{StroockVaradhan} (see also~\cite[Theorem~1.12]{MortersPeres}).
Hence $\P(S_{1,c}>x)<\eps$ for all sufficiently large $x$.
Let $S_{1,c}'=\sup_{1\le s<t\le 2,\,t-s<c}\frac{|B(t)-B(s)|}{\omega(t-s)}$ and note that this is an independent identically distributed copy of $S_1^c$.
Then
\begin{align*}
\P(S_{2,c}>x) & \ge
\P(\max\{S_{1,c},S_{1,c}'\}>x)
\\ & =2\P(S_{1,c}>x)-(\P(S_{1,c}>x))^2
\ge (2-\eps)\P(S_{1,c}>x).
\end{align*}
Combining this with~\eqref{eq:12} yields
\(
\P(S_{1,c}>(1+\eps)^{1/2}x)\le
(2-\eps)^{-1}\P\big(S_{1,c}>x\big)
\)
for $x$ large. By iterating, we obtain
\[
\P(S_{1,c}>(1+\eps)^{n/2}x)\le
(2-\eps)^{-n}\P\big(S_{1,c}>x\big)
\le (2-\eps)^{-n}.
\]
Setting $y=(1+\eps)^{n/2}$,
\[
\P(S_{1,c}>xy)\le  y^{-q},
\]
where $q=2\log(2-\eps)/\log(1+\eps)$.
Shrinking $\eps$ yields $q$ arbitrarily large and completes the proof.
\end{proof}

\section*{Acknowledgements}

The research of Ian Melbourne was supported in part by
FAPESP (Grant CNPJ 43.828.151/0001-45).
The research of Zhe Wang was supported by National Key R\&D Program of China (No. 2023YFA1009200) and NSFC (Grant 12501239).
Zhe Wang is grateful to Prof.\ Zhenxin Liu for helpful discussions and suggestions.
We are very grateful to Nicholas Fleming-V\'azquez for pointing out a refinement (see Remarks~\ref{rem:cf} and~\ref{rem:cf2}) that simplifies and significantly improves the results obtained in this paper and also for mentioning~\cite[Exercise~2.4.8]{StroockVaradhan}.

\end{document}